\documentclass[11pt]{article}
\newlength{\dinwidth}
\newlength{\dinmargin}
\usepackage{amsmath,amssymb,amsfonts}
\usepackage[alphabetic,y2k,initials]{amsrefs}

\usepackage{color, graphics}
\usepackage{float}
\usepackage{tikz}
\usepackage{pdfsync}
\usepackage{hyperref}
\usepackage{multirow}

\usepackage{longtable}
\usetikzlibrary{snakes}
\usepackage[all]{xy}
\DeclareFontShape{OT1}{cmss}{b}{n}{<->ssub * cmss/bx/n}{}
\usetikzlibrary{decorations.markings}
\usetikzlibrary{decorations.pathreplacing}
\usetikzlibrary{arrows,shapes,positioning}
\tikzstyle directed=[postaction={decorate,decoration={markings,    mark=at position #1 with {\arrow{>}}}}]
\tikzstyle rdirected=[postaction={decorate,decoration={markings,   mark=at position #1 with {\arrow{<}}}}]
\allowdisplaybreaks
\tikzset{anchorbase/.style={baseline={([yshift=-0.5ex]current bounding box.center)}}}
\tikzset{anchorbase/.style={baseline={([yshift=-0.5ex]current bounding box.center)}}}

\usepackage{calrsfs}
\DeclareMathAlphabet{\pazocal}{OMS}{zplm}{m}{n}

\newcommand{\Z}{\mathbb{Z}}

\def\mod {\mathrm{mod}}

\def\max{\mathrm{max}}

\newtheorem{theorem}{Theorem}[section]
\newtheorem{proposition}[theorem]{Proposition}
\newtheorem{corollary}[theorem]{Corollary}
\newtheorem{lemma}[theorem]{Lemma}
\newtheorem{remark}[theorem]{Remark}
\newtheorem{example}[theorem]{Example}
\newtheorem{definition}[theorem]{Definition}

\newenvironment{proof}[1][Proof]{\noindent\textit{#1.} }{\hfill$\Box$\newline\medskip}

\numberwithin{equation}{section}

\usepackage{authblk}
\author[1,3]{Vladimir Dragovi\'c}
\author[2,3]{Marko Sto\v si\'c}
\affil[1]{\textsc{The University of Texas at Dallas, Department of Mathematical Sciences}}
\affil[2]{\textsc{CAMGSD, Department of Mathematics, Instituto Superior T\'ecnico, Lisbon}}
\affil[3]{\textsc{Mathematical Institute SANU, Belgrade}}
\affil[1 ]{\texttt{vladimir.dragovic@utdallas.edu}}
\affil[2 ]{\texttt{mstosic@isr.ist.utl.pt}}

\date{}

\title{Billiard Partitions, Fibonacci Sequences, SIP Classes, and Quivers}

\begin{document}

\maketitle

\begin{abstract}
Starting from billiard partitions which arose recently in the description of periodic trajectories of ellipsoidal billiards in $d$-dimensional Euclidean space, we introduce a new type of separable integer partition classes, called type B. We study the numbers of basis partitions with $d$ parts and relate them to the Fibonacci sequence and its natural generalizations. Remarkably, the generating series of basis partitions can be related to the quiver generating series of symmetric quivers corresponding to the framed unknot via knots-quivers correspondence, and to the count of Schr\"oder paths.

\

\noindent
\textsc{MSC2010 numbers}: \texttt{05A15, 05A17, 14H70, 37J35, 16G20, 26C05}

\end{abstract}
	
\tableofcontents

\section{Introduction}
The billiard partitions arose recently \cite{DragRadn2018} in the description of periodic trajectories of ellipsoidal billiards in $d$-dimensional Euclidean and pseudo-Euclidean spaces.  Such partitions uniquely codify the sets of caustics, up to their types, which generate
periodic trajectories. The billiard partitions were then studied in \cite{ADR2021} from a more combinatorial point of view. Developing these ideas from \cite{DragRadn2018} and \cite{ADR2021} further and applying them back to the theory of partitions, in \cite{Andrews2021} a notion of separable integer partitions class (SIP) was defined.

In this paper we introduce a new type of separable integer partitions classes, called type B, which are motivated by the count of basis billiard partitions with $d$ parts. We study the numbers of such partitions and relate them to the Fibonacci sequence and its natural generalizations.

Remarkably, the generating series of basis partitions can be related to the quiver generating series of symmetric quivers corresponding to the framed unknot via knots-quivers correspondence, and to the count of Schr\"oder paths. Quiver generating series and corresponding motivic Donaldson-Thomas invariants appear in \cite{KS}. In the case of a symmetric quiver $Q$, the motivic Donaldson-Thomas invariants can be interpreted as the intersection Betti numbers of the moduli space of all semisimple representations of $Q$ \cite{MR}, or as the Chow-Betti numbers of the moduli space of all simple representations of $Q$ \cite{FR}. Donaldson-Thomas invariants can also be combinatorially constructed in terms of Hilbert schemes \cite{rei2}. In addition, for symmetric quivers it turns out that the motivic Donaldson-Thomas invariants are non-negative integers \cite{E}.

Furthermore, the so-called knots-quivers correspondence was postulated in \cite{KRSS}, stating that for every knot there exists a symmetric quiver such that the generating series of symmetrically colored HOMFLY-PT polynomials of a knot equals, up to variable specializations, the quiver generating series of the corresponding quiver. This deep property  has applications in the enumerative combinatorics, including Fuss-Catalan numbers, and counting of Schr\"oder paths \cite{PSS}, as well as combinatorial construction of invariants counting BPS states for a given knot \cite{KucSul}. Knots-quivers correspondence has also been interpreted in terms of counts of holomorphic curves \cite{EKL}.\\

In Section \ref{sec:Billiard} we review the basics of billiard partition. In Section \ref{sec:SIPA} we recall the definition of the separable integer partitions class, and in Section \ref{sec:SIPB} we introduce separable integer partitions classes of type B. After stripping off technical details, in Section \ref{quivers} we show a surprising result that the generating series of basal billiard partitions matches the quiver generating series of a certain two-node quiver, which in turn is directly related to the count of Schr\"oder paths, and to the framed unknot.

\section{Billiard Partitions}\label{sec:Billiard}
As we mentioned above, billiard partitions uniquely codify the sets of caustics, up to their types, which generate
periodic trajectories.
The period of  a periodic trajectory is the largest part while the winding numbers are the remaining summands of the corresponding partition. The number of parts is equal $d$, which is the dimension of the ambient space. It was proven in \cite{DragRadn2018}, that a  partition $m_0>m_1>\dots>m_{d-1}{>0(=m_d)}$, with $m_{d-1}$ even and not two consecutive $m_i, m_{i+1}$ being both odd, uniquely determines the set of caustics of a given type, which generate periodic trajectories of period $m_0$ and winding numbers $m_i$.  We will refer to such partitions as the Euclidean billiard partitions.

In order to take into account the types of caustics as well, the weighted partitions were introduced, to count the number of possibilities for types of the caustics compatible with the given partition. The following weight function $\phi (n, d, \pi)$ for a given Euclidean billiard partition $\pi$ of length $d$ with the largest part equal $n$ was defined in \cite{ADR2021}:
\begin{align}\label{eq:weights1}
\phi(2m, d, \pi)&=2^{d-1-2s};\\
\phi(2m+1, d, \pi)&=2^{d-2s},
\label{eq:weights2}
\end{align}
where $s$ is the total number of odd parts in $\pi$.

A formal definition of Euclidean billiard partitions, {introduced} in \cite{DragRadn2018} and \cite{ADR2021}, {is given in} Definition \ref{def:ebp} and first examples below  in Example \ref{ex:partnumber}.
Closed forms for the generating functions of such partitions were provided in \cite{ADR2021}, see also Theorem \ref{th:generatingE} below.

\begin{example}
Let $d=2$. Then $\phi(2m+1, 2, \pi)=1$ and $\phi( 2m, 2, \pi)=2$.
\end{example}
\begin{example} Let $d=3$.
For $n=4$, the only possible partition is $\pi=(4, 3, 2)$ and we have $\phi(4, 3, \pi)=1$.

For $n=5$, again there is only one partition $\pi=(5, 4, 2)$, for which $\phi(5, 3, \pi)=2$.

For $n=6$, the partitions are:
$$
\pi_1=(6, 5, 4), \quad
\pi_2=(6, 5, 2), \quad
\pi_3=(6, 4, 2), \quad
\pi_4=(6, 3, 2),
$$
and we have:
$$
\phi(6, 3, \pi_1)=\phi(6, 3, \pi_2)=\phi(6, 3, \pi_4)=1,
\quad
\phi(6, 3, \pi_3)=4.
$$

For $n=7$, the partitions are:
$$
\pi_1=(7, 6, 4), \quad
\pi_2=(7, 6, 2), \quad
\pi_3=(7, 4, 2),
$$
with
$$
\phi(7, 3, \pi_1)=\phi(7, 3, \pi_2)=\phi(7, 3, \pi_3)=2.
$$

\end{example}

We will remind a formal definition of Euclidean billiard partitions.
\begin{definition}[\cite{DragRadn2018}, \cite{ADR2021}]\label{def:ebp}
Let $\mathcal D$  denote the set of all integer partitions into distinct parts where
\begin{itemize}
\item [(E1)] the smallest part is even;
\item [(E2)] adjacent parts are never both odd.
\end{itemize}
We will call elements of $\mathcal D$ the Euclidean billiard partitions.
\end{definition}
Let $p_{\mathcal D}(n)$ denote the number of partitions of $n$ that are in $\mathcal D$.
\begin{align*}
1+\sum_{n\ge1}p_{\mathcal D}(n)q^n = 1&+q^2+q^4+q^5+2q^6+q^7+2q^8+3q^9+3q^{10}\\&+4q^{11}+4q^{12}+6q^{13}
+5q^{14}+9q^{15}+\dots .
\end{align*}
The Euclidean billiard partitions for $n=15$ are given in the following example.
\begin{example}\label{ex:partnumber}
Thus, $p_{\mathcal D}(15)=9$, and the nine partitions of $15$ are $13+2$, $11+4$, $10+3+2$, $9+6$, $9+4+2$, $8+5+2$, $7+6+2$, $6+5+4$, $6+4+3+2$.
\end{example}
Additionally, we shall also need to consider weighting the partitions in $\mathcal D$ as follows. Suppose $\pi \in \mathcal D$ and that $\pi$ has $d$ parts with largest part $n$ and $s$ odd parts. The weight $\phi(n, d, \pi)$ is given by \eqref{eq:weights1} and \eqref{eq:weights2}.

Let  $p_{\mathcal D}(m, n)$ denote the number of partitions of $n$ in $\mathcal D$ with weight $m$. Then:
\begin{align*}
1+\sum_{n\ge1}p_{\mathcal D}(m, n)x^mq^n = 1&+q^2+q^4+q^5+(1+x)q^6+q^7+(1+x)q^8+3q^9\\&+(1+2x)q^{10}+(3+x)q^{11}+(1+2x+x^2)q^{12}\\&+(5+x)q^{13}
+(2+3x+x^2)q^{14}+(6+3x)q^{15}+\dots .
\end{align*}
Referring back to Example \ref{ex:partnumber}, we see that three partitions of $15$ have weight $1$, namely $2+4+9$, $2+6+7$, and $2+3+4+6$ while the remaining six have weight $0$. Thus yielding $(6+3x)$ as the coefficient of $q^{15}$.

One {objective} of \cite{ADR2021} was to provide a closed form for these generating functions. A subset $\mathcal B$ of $\mathcal D$ was identified such that $\pi \in  {\mathcal B}$ if no summand of $\pi$ can be reduced by $2$ with the resulting partition remaining in ${\mathcal D}$. For example $2+4+7$ is not in  ${\mathcal B}$ because
$2+4+(7-2)=2+4+5$ is still in ${\mathcal D}$. On the other hand $2+4+5$ is in ${\mathcal B}$ because $2+4+3$ destroys the order of the parts. The elements of  set ${\mathcal B}$ can be characterized as follows.

\begin{lemma} [\cite{ADR2021}] The set of partitions ${\mathcal B}$ consists of those elements of ${\mathcal D}$   where
\begin{itemize}
\item [(BE1)] the smallest part is $2$;
\item [(BE2)] adjacent parts are never both odd;
\item[(BE3)] the difference between adjacent parts is less or equal 2.
\end{itemize}
\end{lemma}

The set ${\mathcal B}$ is called {\it the basis} of  $\mathcal D$  and its elements are called {\it basal} because of the following fundamental property:
\begin{corollary}[\cite{ADR2021}]\label{corollary:decomp} Every partition  $\pi\in \mathcal D$ with $d$ parts can be uniquely represented by $\pi_1+\pi_2$ where $\pi_1\in {\mathcal B}$ and $\pi_2$ is a partition with not more than $d$ parts each even. Conversely, if a partition is represented as a sum $\pi_1+\pi_2$ where $\pi_1\in {\mathcal B}$ and $\pi_2$ is a partition with not more than $d$ parts each even, then it belongs to $\mathcal D$.
\end{corollary}
\begin{example}
Consider $9+4+2\in \mathcal D$: $ 9+4+2=(5+4+2)+(4+0+0).$
\end{example}

One can notice that $\pi$ and its basal partition $\bar \pi$ have the same weight:
$$
\phi (n, d, \pi)=\phi (n, d, \bar \pi),
$$
since they have the same number of odd parts.

Following \cite{ADR2021}, let us denote by $s(d, n)$ the generating function for those partitions in ${\mathcal B}$ that have exactly $d$ parts and largest part equal $n$.

\begin{example} {We calculate, directly from the definition:}
\begin{align*}
s(5, 8) =&\,\,\, x^2q^{23}+x^2q^{25}+x^2q^{27}=\\
=&\,\, x^{5-1-2\cdot 1}q^{2+3+4+6+8}+x^{5-1-2\cdot 1}q^{2+4+5+6+8}+ x^{5-1-2\cdot 1}q^{2+4+6+7+8}.
\end{align*}
\end{example}
Let us recall the Gaussian polynomials or $q$-binomial coefficients:
\begin{equation*}
\left[
\begin{array}{l}
A\\
B
\end{array}
\right]_q=
\begin{cases}
 0, & \text{if } B<0 \text{ or } B>A\\
\frac{(q;q)_A}{(q;q)_B(q;q)_{A-B}}, & 0\le B\le A
\end{cases}
\end{equation*}
and $(x;q)_N=(1-x)(1-xq)\cdots(1-xq^{N-1}).$

\begin{lemma}[\cite{ADR2021}]\label{lemma:sdn} The functions $s(d, m)$ can be expressed as follows, depending on parity of $m$:
\begin{itemize}
\item[a)]\begin{equation}\label{eq:d2n}
s(d, 2n)= x^{2n-d-1}q^{2n^2-2dn-n+d^2+2d}\left[
\begin{array}{c}
n-1\\
2n-d-1
\end{array}
\right]_{q^2};
\end{equation}
\item[b)]\begin{equation}\label{eq:d2n1}
s(d, 2n+1)=  x^{2n-d}q^{2n^2-2dn+d^2+3n}\left[
\begin{array}{c}
n-1\\
2n-d
\end{array}
\right]_{q^2}.
\end{equation}
\end{itemize}
\end{lemma}

The key observation is that once we know the generating  function for the basal partitions, it is easy to get
the generating function for partitions in $\mathcal D$.

\begin{theorem}[\cite{ADR2021}]\label{th:generatingE} The generating function for the weighted Euclidean billiard partitions has the following formula:
\begin{equation*}
1+\sum_{n\ge1, m\ge 0}p_{\mathcal D}(m, n)q^n =1+\sum_{d=1}^\infty \sum_{n=0}^\infty \frac{s(d, n)}{(q^2;q^2)_d},
\end{equation*}
where
\begin{equation*}
s(d, 2n)= x^{2n-d-1}q^{2n^2-2dn-n+d^2+2d}\left[
\begin{array}{c}
n-1\\
2n-d-1
\end{array}
\right]_{q^2};
\end{equation*}
\begin{equation*}
s(d, 2n+1)= x^{2n-d}q^{2n^2-2dn-n+d^2+3n}\left[
\begin{array}{c}
n-1\\
2n-d
\end{array}
\right]_{q^2}.
\end{equation*}

\end{theorem}

\section{Separable integer partitions}

\subsection{Separable integer partitions of type A}\label{sec:SIPA}

Formalizing further this line of {thought}, in \cite{Andrews2021} a notion of separable integer partitions class (SIP) with modulus $k$  was defined as a subset $\mathcal P$ of all the integer partitions
with a subset $\mathcal B$, called the basis of $\mathcal P$, with the properties:
\begin{itemize}
\item[(i)] for $n\ge 1$ the number of elements  of $\mathcal B$ with $n$ parts is finite;
\item[(ii)] every partition with $n$ parts in $\mathcal P$ can be uniquely presented in the form
\begin{equation}\label{eq:decomposition}
(b_1+p_1)+(b_2+p_2)+\dots+(b_n+p_n),
\end{equation}
$0<b_1\le b_2\le \dots \le b_n$ form a partition $b_1+b_2+\dots+b_n\in \mathcal B$, and $0\le p_1 \le p_2\le \dots \le p_n$ form a partition into $n$ parts, with the only restriction that all parts $p_j$ were divisible by $k$;
\item[(iii)] all the partitions of the form \ref{eq:decomposition} belong to $\mathcal P$.
\end{itemize}

In \cite{Andrews2021} an interesting set of SIP classes was introduced.  Let $\{c_1, c_2, \dots, c_k\}$ be a set of positive integers with $c_r\equiv r (\mod\, k)$ and $\{d_1, d_2, \dots, d_k \}$ be a set of nonnegative integers. Let $\mathcal P$ be the set of all integer partitions
$$
p_1+p_2+\dots+p_j
$$
where $0<p_1\le\dots\le p_j$ and for all $r$, $1\le r\le k$ and each $p_i$ if $p_i\equiv r (\mod\, k)$, then $p_i\ge c_r$ and if $i>1$
$$
p_i-p_{i-1}\ge d_r.
$$
Then $\mathcal P$ is a SIP class modulo $k$. We will say that such a class is of {\it type A}. Its basis $\mathcal B$ was described in Theorem 1 in \cite{Andrews2021}. It consists of all partitions
$$
b_1+\dots + b_j,
$$
where if $b_1\equiv r (\mod\, k)$, then $b_1=c_r$ and for $2\le i\le j$ if $b_i\equiv r (\mod\, k)$ then $d_r\le b_i-b_{i-1}<d_r+k$.

\begin{example} In relation with the G\"ollnitz-Gordon Theorem, in \cite{Andrews2021} the class $\mathcal P_{\mathcal G}$ of all partitions in which the difference between parts is at least 2 and at least 4 between even parts. Then $c_1=1$, $c_2=2$, $d_1=2$, $d_2=3$.
\end{example}
It was observed in \cite{Andrews2021} that billiard partitions do not belong to SIP classes of type A.

\subsection{Separable integer partitions of type B}\label{sec:SIPB}

Before we define a set of SIP classes which contains billiard partitions, we will study the numbers of basal Euclidean billiard partitions with a given number of  parts $d$.
Let
\begin{equation}\label{eq:fibonacci}
 f_0=1,\, f_1=1, \, f_2=2,\, \dots, f_d=f_{d-1}+f_{d-2}.
\end{equation}

\begin{theorem}\label{th:fibonacci}
\begin{itemize}
\item[(a)]
There are $f_d$ basal billiard partitions with $d$ parts.
\item[(b)] There are $f_{d-1}$ basal billiard partitions with $d$ parts with the largest part being an even number.
    \item[(c)] There are $f_{d-2}$ basal partitions with $d$ parts with the largest part being an odd  number.
        \end{itemize}
        \end{theorem}

 \begin{proof}  The proof goes by induction. One can directly verify it for partitions with $d=1, 2, 3$ parts.
The basal partitions with $d+1$ parts can be obtained from those with  $d$ parts in the following ways:
by adding to the list of parts an even number, the consecutive of the largest odd number,
for the partitions with an odd number as the largest part;
by adding to the list of parts an odd number, the consecutive of the largest even number,
for the partitions having an even number as the largest part;
by adding to the list of parts an even number, the consecutive even number of the
largest even number, for the partitions having an even number as the largest part.

Thus, the total number of basal partitions with $d+1$ parts is equal to
$$f_{d-2}+2f_{d-1} = f_d + f_{d-1}.$$

There are $f_{d-1}$ basal partitions with $d+1$ parts which have an odd number as the largest part, while there are
$$f_{d-2}+f_{d-1}=f_d$$
basal partitions with $d+1$ parts which have an even number as the largest part.
\end{proof}

\begin{theorem}\label{TyB} Let $\{c_1, c_2, \dots, c_k\}$ be a set of positive integers with $c_r\equiv r (\mod\, k)$ and $\{d_1, d_2, \dots, d_k \}$ be a set of nonnegative integers. Let $\mathcal P$ be the set of all integer partitions
$$
p_1+p_2+\dots+p_j
$$
where $0<p_1<\dots<p_j$ and for all $r$, $1\le r\le k$ and each $p_i$ if $p_i\equiv r (\mod\, k)$, then $p_i\ge c_r$ and if $i>1$
$$
p_i-p_{i-1}\ge \max\{1, d_r\}.
$$
Then $\mathcal P$ is a SIP class modulo $k$. Its basis $\mathcal B$ consists of all partitions
$$
b_1+\dots +b_j,
$$
where if $b_1\equiv r (\mod\, k)$ then $b_1=c_r$, and for $i\le j$ if $b_i\equiv r (\mod\, k)$ then  {{$b_i\ge c_r$ and}} $\max \{d_r, 1\}\le b_i-b_{i-1}<d_r+k$.
\end{theorem}

\begin{definition} We will denote the SIP classes described in {Theorem \ref{TyB}} as type B of modulus $k$.
Those $j (\mod\, k)$ with $d_j=0$ we will call odd-like, and the number of odd-like residues we will denote $\ell$. In order to indicate the number $\ell$ along with modulus $k$ we will say that a given partition class of type B is of modulus $k_{\ell}$.

\end{definition}

\begin{example} The billiard partitions belong to a class of type B. We consider those which satisfy $E2$. The basis satisfies the axioms $BE2$ and $BE3$. In this case $c_1=1$, $c_2=2$, $d_1=0$, $d_2=1$, $k=2$, $\ell=1$, and modulus is $2_1$.
\end{example}

\begin{example}\label{ex:32} We consider the partition class $\mathcal P_{3,2}$ of type B with $k=3$, $c_1=1$, $c_2=2$, $c_3=3$, $d_1=0$, $d_2=0$, $d_3=1$, $\ell=2$. Thus, the modulus is $3_2$. If $f_d$ denotes the number of basal elements of $\mathcal P_{3,2}$ with $d$ parts,
then,
\begin{equation}\label{eq:32}
 f_{d+2}=2f_{d+1}+f_d.
\end{equation}
Here $f_1=1,\, f_2=3$, if we consider partitions with the smallest part equal to 3. If the smallest part is equal to 1 or 2, then $f_1=1,\, f_2=2$. If we don't make any restrictions on the smallest part, then $f_1=3,\, f_2=7$.
\end{example}
\begin{example}\label{ex:31} We consider the partition class $\mathcal P_{3,1}$ of type B with $k=3$, $c_1=1$, $c_2=2$, $c_3=3$, $d_1=0$, $d_2=1$, $d_3=1$, $\ell=1$. Thus, here the modulus is $3_1$. If $f_d$ denotes the number of basal elements of $\mathcal P_{3,1}$ with $d$ parts,
then,
\begin{equation}\label{eq:31}
f_{d+2}=2f_{d+1}+2f_d.
\end{equation}
Here $f_1=1,\, f_2=3$, if we consider partitions with the smallest part equal to 2 or 3. If the smallest part is equal to 1  then $f_1=1,\, f_2=2$. If we don't make any restrictions on the smallest part, then $f_1=3,\, f_2=8$.
\end{example}

\begin{theorem} Let $\mathcal P$ be a partition class of type B of modulus $k_{\ell}$, and suppose that all non-zero $d_j=1$. Let $f_d$ denote the number of elements of the basis $\mathcal B$ of $\mathcal P$ with $d$ parts. Then $f_d$ satisfies the difference equation
\begin{equation}\label{eq:fkl1}
f_{d+2}=T_{k,\ell}f_{d+1}-R_{k,\ell}f_d,
\end{equation}
where
\begin{equation}\label{eq:fkl2}
T_{k,\ell}=k-1, \quad R_{k,\ell}=\ell-k.
\end{equation}
Here $f_1=1,\, f_2=k$, if we consider partitions with the smallest part being fixed and not odd-like. If the smallest part is fixed and odd-like  then $f_1=1,\, f_2=k-1$. If there are no  restrictions on the smallest part, then $f_1=k,\, f_2=k^2-\ell$.
\end{theorem}

\begin{proof} One can easily see that
\begin{equation}\label{eq:fd1}
f_{d+1}=(k-1)\cdot a_d + k\cdot b_d, \quad d\ge 1,
\end{equation}
where
\begin{equation}\label{eq:fd2}
\begin{aligned}
a_{d+1}&=(\ell-1)\cdot a_d+\ell\cdot b_d,\\
b_{d+1}&=(k-\ell)(a_d+b_d).
\end{aligned}
\end{equation}
Here $a_d$ denotes the number of elements of the basis with $d$ parts ending on an ``odd digit", and $b_d$ the number of the rest of the elements
 of the basis with $d$ parts. Thus,
 $$
 f_d=a_d+b_d.
 $$
By substituting \eqref{eq:fd1} and \eqref{eq:fd2} into \eqref{eq:fkl1}, one gets \eqref{eq:fkl2}.
\end{proof}

The sequences satisfying the general difference equations of the form \eqref{eq:fkl1} are called {\it Lucas sequences}.
\begin{example} For the billiard partitions, where $k=2, \ell=1$ we get the Fibonacci sequence $T_{2,1}=1$, $R_{2,1}=-1$ as in \eqref{eq:fibonacci}.
\end{example}

\begin{example} For the partition classes of type B of modulus $3_1$ as in Example \ref{ex:31} we get the  sequence $T_{3,1}=2$, $R_{3,1}=-2$ as in \eqref{eq:31}.
\end{example}

\begin{example} For the partition classes of type B of modulus $3_2$ as in Example \ref{ex:32} we get the  sequence $T_{3,2}=2$, $R_{3,2}=-1$ as in \eqref{eq:32}. So-called Pell's numbers are generated by Lucas sequences with the same coefficients $2, -1$ and with the initial data $f_0=0, f_1=1$.

\end{example}

We can generalize the last examples in the following ways.

\begin{example} Consider a partition class of type B of modulus $k_1$. Then $T_{k, 1}=k-1=-R_{k,1}$ in the recurrence relation \eqref{eq:fkl1} satisfied by the sequence $f_d$ of the numbers of basal partitions with $d$ parts.

 Consider a partition class of type B of modulus $k_{k-1}$. Then $T_{k, k-1}=k-1$ and $R_{k,k-1}=-1$ in the recurrence relation \eqref{eq:fkl1} satisfied by the sequence $f_d$ of the numbers of basal partitions with $d$ parts.

{Taking into account that for billiard partitions $(T, R)=(1,-1)$ corresponds to  the Fibonacci sequence \eqref{eq:fibonacci},
there are two natural generalizations: $(T, R)=(k-1,-(k-1))$ and  $(T, R)=(k-1,-1)$.}
 \end{example}

\subsection{Computation of number of refined basal elements}
Here we give direct computation of the number of basal elements, with the refinement, via generating series. We begin with the case of billiard partitions.

Let $s(d,m)$ be the generating function of basal billiard partitions with $d$ parts, and largest part $m$. To simplify notation, we simply set $x=1$, i.e. we don't keep {explicitly} the track of the weight, since it follows straightforward from $d$ and $m$. Then as in \cite{ADR2021} we have the obvious recursions:
\begin{eqnarray}
s(d,2n)&=&q^{2n}(s(d-1,2n-2)+s(d-1,2n-1)),\\
s(d,2n+1)&=&q^{2n+1}s(d-1,2n).
\end{eqnarray}
This gives
\begin{equation}\label{pom1}
s(d,2n)=q^{2n}s(d-1,2(n-1))+q^{4n-1}s(d-2,2(n-1)).
\end{equation}
Let
$$s_n(z)=\sum_{d\in\Z} s(d,2n) z^d.$$
Note that the above sum is in fact finite, since $s(d,2n)$ can be nonzero only when $d+1\le 2n \le 2d$. From (\ref{pom1}) we get:
$$s_n(z)=(q^{2n}z+q^{4n-1}z^2)s_{n-1}(z)=q^{2n}z(1+q^{2n-1}z) s_{n-1}(z),\quad n\ge 2,$$
with $s_1(z)=q^2z$. Therefore, for $n\ge 1$:
$$s_n(z)=q^2z\prod_{i=2}^n q^{2i}z(1+q^{2i-1}z) =q^{n(n+1)}z^n\prod_{i=1}^{n-1} (1+q^3z q^{2(i-1)}),$$
which by the quantum binomial formula gives:
$$s_n(z)=q^{n(n+1)}z^n\sum_{k=0}^{n-1}q^{3k}z^k q^{k^2-k}\left[
\begin{array}{c}
n-1\\
k
\end{array}
\right]_{q^2}=\sum_{d=n}^{2n-1} q^{2n^2-2dn-n+d^2+2d}\left[
\begin{array}{c}
n-1\\
d-n
\end{array}
\right]_{q^2}\, z^d,$$
where we introduced the summation index $d=n+k$. This gives the desired explicit formula for {$s(d,2n)$.}

\paragraph{Computations for basal elements of $\mathcal{P}_{3,2}$}
 Let $t(d,m)$ denote the generating function for the basal partitions from $\mathcal{P}_{3,2}$ with $d$ parts, and largest part $m$. Then from the definition we have the following recursion relations:
\begin{eqnarray}
t(d,3n)&=&q^{3n}(t(d-1,3n-3)+t(d-1,3n-2)+t(d-1,3n-1)), \label{t1}\\
t(d,3n+1)&=&q^{3n+1}(t(d-1,3n-1)+t(d-1,3n)), \label{t2}\\
t(d,3n+2)&=&q^{3n+2}(t(d-1,3n)+t(d-1,3n+1)). \label{t3}
\end{eqnarray}
By using (\ref{t3}), the expression (\ref{t1}) becomes
\begin{equation}\label{t4}
t(d,3n)=q t(d,3n-1)+q^{3n}t(d-1,3n-1).
\end{equation}
Further replacing of (\ref{t4}) in (\ref{t2}) gives:
\begin{equation}\label{t5}
t(d,3n+1)=(q^{3n+1}+ q^{3n+2}) t(d-1,3n-1)+q^{6n+1}t(d-2,3n-1).
\end{equation}
Finally, by (\ref{t4}) and (\ref{t5}), the expression (\ref{t3}) gives:
\begin{equation}\label{t6}
t(d,3n+2)=q^{3n+3} t(d-1,3n-1)+q^{6n+2}(1+q+q^2)t(d-2,3n-1)+q^{9n+3}t(d-3,3n-1).
\end{equation}

Let \begin{eqnarray}
c_n(z)&=&\sum_{d\in\Z}t(d,3n+2)z^d\label{c}\\
b_n(z)&=&\sum_{d\in\Z}t(d,3n+1)z^d\label{b}\\
a_n(z)&=&\sum_{d\in\Z}t(d,3n)z^d.\label{a}
\end{eqnarray}
Then from (\ref{t6}) we have
$$c_n(z)=(q^{3n+3}z+q^{6n+2}(1+q+q^2)z^2+q^{9n+3}z^3)c_{n-1}(z),\quad\textrm{  for }\quad n\ge 2.$$  The initial condition is
$$c_1(z)=\sum_{d\in\Z}t(d,5)z^d=q^8z^2+q^{12}z^3.$$
Therefore for $n\ge 1$, we have
\begin{equation}\label{pom2}\sum_{d\in\Z}t(d,3n+2)z^d=c_n(z)=(q^8z^2+q^{12}z^3)\prod_{i=2}^n(q^{3i+3}z+q^{6i+2}(1+q+q^2)z^2+q^{9i+3}z^3).\end{equation}

From (\ref{c})-(\ref{a}) and by (\ref{t4})
 and (\ref{t5}), we have
$$\sum_{d\in\Z}t(d,3n)z^d=a_n(z)=(q+q^{3n}z)c_{n-1}(z),$$
$$\sum_{d\in\Z}t(d,3n+1)z^d=b_n(z)=((q^{3n+1}+q^{3n+2})z +q^{6n+1}z^2)c_{n-1}(z).$$

 In the lack of explicit quantum multinomial formula, {let us} focus on the classical limit
 $q\to 1$ in $t(d,3n+2)$ i.e. $c_n(z)$. Then from (\ref{pom2}) we have
   \begin{equation}\label{pom3}c_n(z)_{\mid_{q=1}}=(z^{n+1}+z^{n+2})(1+3z+z^2)^{n-1}= (z^{n+1}+z^{n+2})\sum_{i+j+k=n-1}\frac{(n-1)!}{i!j!k!}3^iz^{i+2j}.\end{equation}

We note that similar computations where obtained in \cite{Andrews2021}, for similar classes of separable integer partitions with recursion relations of order 3.


\section{Relationship with quiver generating series}\label{quivers}

\subsection{Generating series for basal billiard partitions}\label{basals}
Using variable $a$ instead of $x$ here, the generating series for the basal billiard partitions is given by:
\begin{equation*}
p_{\mathcal B}(a,q)=1+\sum_{n\ge1, m\ge 0}p_{\mathcal B}(m, n)q^n =1+\sum_{d=1}^\infty \sum_{n=0}^\infty s(d, n),
\end{equation*}
where
\begin{equation*}
s(d, 2n)= a^{2n-d-1}q^{2n^2-2dn-n+d^2+2d}\left[
\begin{array}{c}
n-1\\
2n-d-1
\end{array}
\right]_{q^2};
\end{equation*}
\begin{equation*}
s(d, 2n+1)= a^{2n-d}q^{2n^2-2dn-n+d^2+3n}\left[
\begin{array}{c}
n-1\\
2n-d
\end{array}
\right]_{q^2}.
\end{equation*}
Let us focus first on the even ones:
\begin{equation*}
p_{\mathcal B}^{even}(a,q) =1+\sum_{d=1}^\infty \sum_{n=0}^\infty s(d, 2n).
\end{equation*}
Instead of simply using $s(d,2n)$, let us use further stratification $S(d,2n)$, where we also keep track of the largest element $n$:
\begin{equation*}
S(d, 2n)= x^{n-1} a^{2n-d-1}q^{2n^2-2dn-n+d^2+2d}\left[
\begin{array}{c}
n-1\\
2n-d-1
\end{array}
\right]_{q^2}.
\end{equation*}
Then let
\begin{equation*}
P_{\mathcal B}^{even}(a,q) =1+\sum_{d=1}^\infty \sum_{n=0}^\infty S(d, 2n)=\sum_{n\le d<2n}x^{n-1} a^{2n-d-1}q^{2n^2-2dn-n+d^2+2d}\left[
\begin{array}{c}
n-1\\
2n-d-1
\end{array}
\right]_{q^2}.
\end{equation*}
After change of variables:
\begin{eqnarray*}
i &=&2n-d-1,\\
j&=& d-n,
\end{eqnarray*}
we get:
\begin{eqnarray}
P_{\mathcal B}^{even}(q,a,x) &=&\sum_{i,j\ge 0}x^{i+j} a^{i}q^{i^2+2ij+2j^2+3i+5j+2}\left[
\begin{array}{c}
i+j \nonumber\\
i
\end{array}
\right]_{q^2}=\\
&=&\quad q^2 \sum_{i,j\ge 0}x^{i+j} a^{i}q^{i^2+2ij+2j^2} q^{3i+5j}\frac{(q^2;q^2)_{i+j}}{(q^2;q^2)_{i}(q^2;q^2)_{j}}. \label{eq:fS1}
\end{eqnarray}
This is now presented in the form of a quiver generating series.
\begin{remark}The overall factor $q^2$ comes from the fact that the minimal element in billiard partitions is 2, and can be neglected from now on.
\end{remark}

\subsection{Quiver generating series}
For a symmetric quiver $Q$, with $m$ nodes and adjacency matrix $C$, the corresponding quiver generating series is given by:
\begin{equation}\label{qs1}
\bar{P}_{Q}(x_1,\ldots,x_m;q)=\sum_{d_1,...,d_m\ge 0} (-q)^{\sum_{i,j}C_{ij}d_id_j} \frac{x_1^{d_1}\ldots x_m^{d_m}}{\prod_{i=1}^m (q^2;q^2)_{d_i}}.
\end{equation}

We also say that a power series $P$ is in the quiver form if it is in the form
\begin{equation}\label{qs2}
P_{Q}(x_1,\ldots,x_m;q)=\sum_{d_1,...,d_m\ge 0} (-q)^{\sum_{i,j}C_{ij}d_id_j} \frac{(q^2;q^2)_{d_1+\ldots+d_m}}{\prod_{i=1}^m (q^2;q^2)_{d_i}}{x_1^{d_1}\ldots x_m^{d_m}}.
\end{equation}

The knots-quivers correspondence, introduced in \cite{KRSS}, relates the generating series of the colored HOMFLY-PT polynomials of a given knot,
with the quiver generating series for a corresponding quiver. In particular, the two quiver forms from above correspond to the unreduced and reduced version
of colored HOMFLY-PT invariants, respectively.\\

Let $Q$ be the following two-vertex quiver:
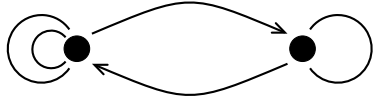
\begin{center}
\begin{tikzpicture}[scale=1,>=angle 45]
\draw[thick,->] (-1.3,0.2) .. controls (0,0.75)..(1.3,0.2);
\draw[thick,->] (1.3,-0.2) .. controls (0,-0.75)..(-1.3,-0.2);
\draw [fill] (-1.5,0) circle (0.17cm);
\draw [fill] (1.5,0) circle (0.17cm);
\draw[thick] (-1.6,-0.26) arc (325:35:0.45cm);
\draw[thick] (-1.65,-0.17) arc (320:40:0.25cm);
\draw[thick] (1.6,-0.26) arc (-145:145:0.45cm);
\end{tikzpicture}
\end{center}
The corresponding adjacency matrix $C$ is $2\times 2$ matrix with $(i,j)$ entry being the number
of arrow from vertex $i$ to vertex $j$, and is therefore given by
\begin{equation}\label{qc2}
C=\left[\begin{array}{cc}
2&1\\
1&1
\end{array}\right].
\end{equation}

One of the main results of this paper is that the generating series of the even basal billiard partitions can be recognised to be in the quiver form (\ref{qs2}). Remarkably, the corresponding quiver is precisely the quiver $Q$.

\begin{proposition}\label{prop2}
Let $P_{\mathcal B}^{even}(q,a,x)$ be a generating series for even basal billiard partitions, as in Section \ref{basals}.
Let $Q$ be the two-vertex quiver from above. Then after setting $$x_1=q^5x,\quad x_2=-aq^3x,$$
we get that the quiver generating series of $Q$ matches the generating series for even basal billiard partitions
\begin{equation}
P_{\mathcal B}^{even}(q,a,x)=P_{Q}(x_1,x_2;q)_{| x_1=q^5x,\, x_2=-aq^3x} .
\end{equation}
\end{proposition}

This particular two-vertex quiver $Q$ from Proposition \ref{prop2}, and its quiver generating series, turns out to be interesting for various reasons. First of all, under the knots-quivers correspondence, the quiver generating series (\ref{qs1}) of this quiver corresponds to the series of the unreduced colored HOMFLY-PT polynomials of the 1-framed unknot:
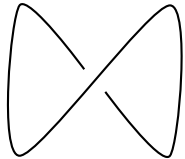
\begin{center}
\begin{tikzpicture}[scale=1,>=angle 45]
\draw[thick] plot [smooth] coordinates{(0.86,1.15) (0,2) (0,0) (2,2) (2,0) (1.14,0.85)};
\end{tikzpicture}
\end{center}
And secondly, it was shown in \cite{PSS} that from the quiver generating series of this quiver, one can naturally count
the Schr\"oder paths -- the lattice paths in the first quadrant below the diagonal $y=x$, so that each step can be either to the right,
up, or diagonal:

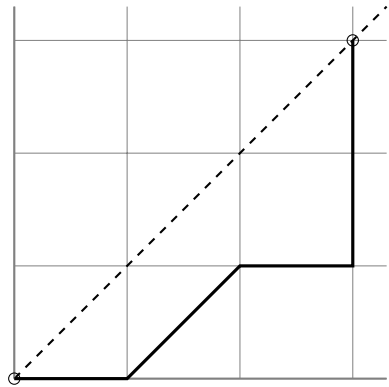
\begin{center}
\begin{tikzpicture}[scale=1.5]
\draw[step=1,gray,very thin] (0,0) grid (3.3,3.3);
\draw[thick,dashed] (0,0)--(3.3,3.3);
\draw[thick,gray] (0,0)--(3.3,0);
\draw[thick,gray] (0,0)--(0,3.3);b
\draw[very thick] (0,0) --(1,0)--(2,1)--(3,1)--(3,3);
\draw (0,0) circle (0.05);
\draw (3,3) circle (0.05);
\end{tikzpicture}
\end{center}

Each such path can be counted with its weight: to a Schr\"oder path starting at $(0,0)$ and ending at $(n,n)$, with $D$ diagonal steps, and
such that the area between the path and the diagonal $y=x$ is equal to $A$, we associate a weight: $a^D q^{2A} x^n$. For example, the path from the figure above has  weight
$a^1 q^6 x^3$.


A surprising result from \cite{PSS} shows that the quotient:
$$\frac{\bar{P}_Q(x_1 q,x_2 q,q)}{\bar{P}_Q(x_1 q^{-1},x_2 q^{-1},q)},$$
for the above two-vertex quiver $Q$, after setting $x_1=x$, $x_2=ax$, becomes equal, up to an overall factor, to the generating series of weighted Schr\"oder paths.\\

\subsection{Conclusion}
The results of this paper suggest a direct relationship between billiard partitions and quiver generating series for a specific quiver, and count of Schr\"oder paths. This rises questions about the relations between various different counts and combinatorial constructions arising from billiard partitions, quiver generating series and knots-quivers correspondence. In particular, one interesting question would be to understand whether Donaldson-Thomas invariants, arising from a product form of the quiver generating series, have some counterpart in billiard partitions. More detailed analysis and study of all these possible relationships is postponed for the future work.

As a final, low-road comment, it is interesting to note that the relationship described in this paper contains a curious connection between the two famous combinatorial sequences -- Fibonacci numbers, and Catalan numbers. For billiard partitions, the count is always a natural refinement of Fibonacci numbers, as seen in \cite{ADR2021} and also in this paper, e.g. Theorem \ref{th:fibonacci}. On the other hand, in the counting combinatorics motivated by knots-quivers correspondence and related BPS count (see \cite{PSS} or \cite{KucSul}), Catalan numbers appear naturally, together with their refinements, like the count of Schr\"oder paths. Therefore, it is quite amusing to see that the two different enumerations that are natural refinements of Fibonacci numbers and Catalan numbers, are in fact directly related.

\subsection*{Acknowledgment} {The authors are grateful to the referees for helpful remarks and suggestions.
This work has been partially supported by the Science Fund of the Republic of Serbia, Projects no. 7744592, MEGIC -- ``Integrability and Extremal Problems in Mechanics, Geometry and Combinatorics" (V.D.) and no. 7749891, GWORDS -- ``Graphical Languages" (M.S.). V.D. was also partially supported by the Simons
Foundation grant no. 854861, and M.S. was also partially supported by Funda\c{c}\~ao para a Ci\^encia e a Tecnologia (FCT) through
Exploratory Grant  EXPL/MAT-PUR/0584/2021, and CEEC grant with DOI no. 10.54499/2020.02453.CEECIND/CP1587/CT0007.

\end{document}